\documentclass[12pt]{amsart}
 \usepackage{amsmath,amssymb,amsthm,mathrsfs}
\usepackage{epsfig,color,mathtools}

\definecolor{pp}{rgb}{.5,0,.8}
\definecolor{pr}{rgb}{.2,0.5,.8}

\makeatother

\theoremstyle{definition}

\def\fnum{equation} 
\newtheorem{Thm}[\fnum]{Theorem}

\newtheorem{Que}[\fnum]{Question}
\newtheorem{Lem}[\fnum]{Lemma}

\newtheorem{Rem}[\fnum]{Remark}

\numberwithin{equation}{section}

\newcommand{\Int}{\operatorname{Int}}

\def\RR{{\mathbb{R}}}

\def\CC{{\mathbb C }}

\newcommand{\e}{{\text {e}}}

\newcommand{\Area}{{\text {Area}}}

\title{Minimal surfaces with rapid area growth}

\author{Tobias Holck Colding}%
\address{MIT, Dept. of Math.\\
77 Massachusetts Avenue, Cambridge, MA 02139-4307.}
\author{Francisco Mart\'in}
\author{William P. Minicozzi II}
\address{Department of Geometry and Topology, University of Granada}

\thanks{The  first and third authors
were partially supported by NSF  DMS Grants   2405393 and 2304684. The second author was partially supported by the grant PID2024-156031NB-I00 and by the  IMAG-Maria de Maeztu grant CEX2020-001105-M, both funded by MICIU/AEI/10.13039/501100011033.}

\email{colding@math.mit.edu, fmartin@ugr.es and minicozz@math.mit.edu}

\begin{document}

\maketitle

\begin{abstract}
We give examples of proper minimal immersions in Euclidean space with very rapid area growth. The first is a proper embedding into $\RR^4$ that yields a stable minimal surface, while the second is a proper  immersion into $\RR^3$.

These results are motivated by \cite{CM1} that shows that proper minimal submanifolds confined in space satisfy strong structural constraints.
\end{abstract}


\section{Introduction}

David Hilbert \cite{H} proved, in 1901, a celebrated impossibility result: there is no complete surface in 
$\RR^3$
  with constant Gaussian curvature $-1$. More generally, Hilbert's argument shows that there is no complete immersed surface in $\RR^3$
  whose Gaussian curvature is bounded above by a negative constant.
In the opposite direction, Nadirashvili \cite{N} constructed examples of complete immersed minimal surfaces contained in the unit ball of $\RR^3$
  with everywhere negative curvature. These should be compared with the earlier examples of Jorge--Xavier \cite{JX}, who constructed nontrivial immersed minimal surfaces contained in a slab in $\RR^3$.
 
Here we are interested in a related but distinct question: the existence of complete proper\footnote{An immersion is proper if its intersection with every compact set is compact.} minimal immersions into $\RR ^n$
  with at least exponential area growth. Any simply connected surface whose curvature is bounded above by a negative constant necessarily has exponential area growth, so one may regard exponential area growth together with minimality as a weaker analogue of the geometric conditions considered by Hilbert.

Neither the examples of Nadirashvili nor those of Jorge--Xavier are proper immersions. In fact, the half-space theorem of Hoffman--Meeks, \cite{HM}, implies that 
any proper minimal immersion contained in a half-space in $\RR^3$ must be planar.

We present two examples of proper minimal immersions into Euclidean space exhibiting exponential area growth. The first is a proper embedding into $\RR^4$ that gives rise to a stable minimal surface, while the second is a proper immersion into $\RR^3$ inspired by \cite{MM04}. We also indicate how both of these constructions can be adapted to produce examples with even faster growth.
  
Our interest in these questions stems from \cite{CM1} that shows that proper minimal submanifolds confined in space satisfy strong structural restrictions. In particular, by \cite{CM1}, any proper minimal immersion whose height grows sublinearly must have Euclidean volume growth. More generally, \cite{CM1} shows that a minimal submanifold of any dimension that is contained in a slab satisfies a volume doubling property. If the slab condition holds on all sufficiently large scales, then the doubling estimate can be iterated to obtain a polynomial upper bound for the volume growth.  On the other hand, a submanifold with exponential volume growth  does not satisfy a volume doubling property and therefore is not confined to slabs across multiple scales.  

The examples here are complete and have at least exponential volume growth and, thus, arbitrarily large volume doubling.  Compact examples with boundary in $\RR^3$ that have
arbitrarily large volume doubling were constructed in \cite{CM2,HW,MM04,MM05}.

\section{Stable minimal surfaces in $\RR^4$ with rapid area growth}

In this section, we present two examples of embedded stable minimal surfaces in $\mathbb{R}^4$ with very rapid area growth. Both are realized as graphs of holomorphic functions. The first exhibits exponential area growth, while the second has even faster Gaussian area growth.

\subsection{Example with exponential area growth}

\begin{Thm}   \label{t:exA}
Let $\Sigma = \{(z,f(z)) : z \in \mathbb{C}\}$  
be the graph in $\mathbb{C}^2$ of the holomorphic function
$f(z)=\sin(\e^z)$.  
Then $\Sigma$ is a smooth stable minimal surface in
$\mathbb{R}^4=\mathbb{C}^2$ 
and there exist constants $c>0$ and $r_0>0$ such that for all $r\ge r_0$,
\begin{align}
\Area\, (B_r\cap \Sigma)\ge \e^{c\,r}\,  .
\end{align}
\end{Thm}

\begin{proof}
Since $\Sigma$ is the graph of a holomorphic function, it is a complex curve in $\mathbb{C}^2$ and therefore area-minimizing by \cite{W}.  In particular, it is automatically a stable minimal surface.

Using the Cauchy--Riemann equations, we obtain
\begin{align}
\Area\, (B_r\cap \Sigma)
=
\int_{\Omega_r}
\left(
1+\left|f'(z)\right|^2
\right)\,dA\,  ,
\end{align}
where 
$\Omega_r
=
\{z\in\CC:\ |z|^2+|f(z)|^2\le r^2\}$, $f'(z)$ is the complex derivative, 
and $dA$ denotes Lebesgue measure on the $z$-plane.  
For $n=1,2,3,\dots$,  define $z_n=\log(n\,\pi)$ and observe that 
\begin{align}
|f(z+z_n)|
&= \left|\sin\!\left(\pi \,n \,\e^{z}\right)\right| = \left|\sin\!\left(\pi\, n\, (\e^{z}-1)\right)\right|\,  ,\\
|f'(z+z_n)|&= \left|\pi\, n\, \e^z \cos\!\left(\pi\, n \,\e^z\right)\right| = \left|\pi\, n\, \e^z \cos\!\left(\pi\, n (\e^z-1)\right)\right| .
\end{align}
Since $|\e^z-1|\leq |z|\,\e^{|z|}$  for all $z$, we have that as long as $|z|\leq \frac{1}{16\, \pi \,n}$ then 
\begin{align}
	|\e^z-1|\leq   \frac{ \e^{ \frac{1}{16\, \pi \, n}}}{16\, \pi \,n}  < \frac{\e}{16\, \pi \,n} < \frac{1}{4\, \pi \, n}  \, .  \label{e:bdw}
\end{align}
For each $n$, let $D_n$ be the disk
$D_n=
\left\{ z+z_n:\
|z|<\frac{1}{16\, \pi \,n}
\right\}$.    We will need that $D_{n_1}\cap\ D_{n_2}=\emptyset$ for $n_1\ne n_2$.  We will show that  $D_n\cap D_{n+1}=\emptyset$; the general case is similar.  Suppose therefore that $w_n\in D_n$ and $w_{n+1}\in D_{n+1}$.  Observe that 
\begin{align}
|z_{n+1}-z_n|=\log \left(1+\frac{1}{n}\right)\geq \frac{1}{2\,n}\,  .
\end{align}
Therefore, by the triangle inequality 
\begin{align}
|w_{n+1}-w_n|\geq |z_{n+1}-z_n|-|w_{n+1}-z_{n+1}|-|w_n-z_n|\geq \frac{1}{2\,n}-\frac{2}{16\,\pi\,n}>\frac{1}{4\,n}\,  .
\end{align}
This shows that $D_{n+1}\cap D_n=\emptyset$.  

Using that $|\sin w| \leq \e^{|w|}$ for any $w \in \CC$, we see that
\begin{align}
	|f (z+z_n) | =  \left|\sin\!\left(\pi\, n\, (\e^z-1)\right)\right| \leq \e^{ \frac{1}{4}} < 2  
	\, .
\end{align}
Using that $\cos' (w) = - \sin (w)$ and $|\cos (n \pi)| = 1$,  we also have on $D_n$ that
\begin{align}
	|f'|\geq \pi \, n \, \e^{ - \frac{1}{4}} \, \left| 1 - \frac{1}{4} \, \e^{ \frac{1}{4}} \right|  \geq \frac{\pi \, n}{4}  \, . 
\end{align}
 It follows that if $r\geq \log (n\,\pi)+2$, then $D_k\subset \Omega_r$ for $k=1,\cdots, n$ and 
\begin{align}
\Area\, (B_r\cap\Sigma) &\geq \sum_{k=1}^n\int_{D_k}(1+|f'(z)|^2)\,dA\geq \sum_{k=1}^n\int_{D_k} \frac{k^2\,\pi^2}{16} 
=    \frac{n\, \pi}{16^3} \, .
\end{align}
The claim easily follows from this.
\end{proof}

It is interesting to note that the area growth for intrinsic balls is also at least exponential.  To see this, observe that the disks $D_n$ in the proof are centered on the real axis, the length distortion is bounded along  the real axis, and the 
 images of the disks   have uniformly bounded diameter.

\subsection{Example with Gaussian area growth}

Let $\Sigma = \{(z,f(z)) : z \in \mathbb{C}\}$  
be the graph in $\mathbb{C}^2$ of the holomorphic function
$f(z)=\sin(\e^{z^2})$.  

\begin{Thm}
The graph $\Sigma$ is a smooth stable minimal surface in
$\mathbb{R}^4=\mathbb{C}^2$ 
and there exist constants $c>0$ and $r_0>0$ such that for all $r\ge r_0$,
\begin{align}
\Area\, (B_r\cap \Sigma)\ge \e^{c\,r^2}\,  .
\end{align}
\end{Thm}

\begin{proof}
It follows as in the proof of Theorem \ref{t:exA} that  $\Sigma$ is a smooth stable minimal surface.

For $n=2,3,\dots$, define points
$$z_n=\sqrt{\log(n\,\pi)}$$ and disks
$D_n=
\left\{ z+z_n:\
|z|<\frac{\delta}{n \,\sqrt{\log n}}
\right\}$ for a fixed but sufficiently small $\delta>0$.    
Observe that 
\begin{align}
|f(z+z_n)|
&= \left|\sin\!\left(\pi \,n \,\e^{z\,(z+2\,z_n}\right)\right| = \left|\sin\!\left(\pi\, n\, (\e^{z\,(z+2\,z_n)}-1)\right)\right|\,  ,\\
|f'(z+z_n)|&= \left|2\,\pi\, n\, (z+z_n)\e^{z\,(z+2\,z_n)} \cos\!\left(\pi\, n \,\e^{z\,(z+2\,z_n}\right)\right|\notag\\
&= \left|2\,\pi\, n\,(z+z_n)\, \e^{z\,(z+2\,z_n)} \cos\!\left(\pi\, n (\e^{z\,(z+2\,z_n)}-1)\right)\right| .
\end{align}
The rest of the proof proceeds as the proof of Theorem \ref{t:exA} with obvious modifications.  
\end{proof}

 \vskip1mm
 While any rapidly growing entire function might seem like a  candidate for producing graphs with rapid area growth, the relevant condition is more delicate. 
A useful contrasting example is
$$
f(z)=\e^z .
$$
This is a rapidly growing entire function with
$
\max_{|z|\le R}|\e^z|=\e^R.
$
However, we will see that the graph $\Gamma_{\e^z}$ of $\e^z$ has quadratic area growth.

Indeed, writing $z=x+iy$, one has $f'(z)=f(z)=\e^z$, and hence
$ 
|f(z)|=|f'(z)|=\e^x.
$ 
The domain $\Omega_R 
=
\{(x,y): x^2+y^2+\e^{2x}\le R^2\}
$ is contained in 
$$
|y|\le R {\text{ and }}
-R\le x\le \frac{1}{2} \, \log R.
$$
Therefore, for $R\ge 2$,
\begin{align}
\operatorname{Area}(B_R\cap \Gamma_{\e^z})
&=
\int_{\Omega_R}(1+\e^{2x})\,dx\,dy        
 \le
\int_{-R}^{\frac{1}{2} \,  \log R}\int_{-R}^{R}(1+\e^{2x})\,dy\,dx    \notag    \\
&=
2R\int_{-R}^{\frac{1}{2} \, \log R}(1+\e^{2x})\,dx       \le
C \, R^2 \,  .
\label{eq:ez-polynomial-area}
\end{align}
 Hence, rapid growth alone of $f(z)$ is not enough to guarantee rapid area growth of its graph.  This is because when $f$ is large, the graph leaves the extrinsic ball.  
 
 The rapid area growth in the examples comes from having enough regions where $f$ is bounded, but its derivative is large. 
  Functions modeled on $\sin(\e^z)$   are   well-suited to this purpose.  Replacing $\e^z$ and $\e^{z^2}$  inside the sine by even  faster-growing functions should yield correspondingly rapid lower bounds on the area growth, so long as there are enough regions where the function is bounded and the derivative is large.


\section{Proper minimal surfaces in $\mathbb R^3$ with exponential area growth}
\label{sec:proper-R3-exponential-growth}

In this section we prove the existence of a complete proper minimal immersion
in $\mathbb R^3$ whose extrinsic area growth is at least exponential.  

\subsection{Overview of the construction}

Before giving the details, we briefly describe the three ideas used in the
construction.  The argument is a standard ``push the boundary outward and take a
limit'' construction, but with one additional feature: at each stage we also
insert a controlled amount of area.

The first ingredient is the properness lemma of Mart\'{\i}n--Morales
\cite[Lemma~2]{MM05}.  In informal terms, this lemma says the following.  Suppose
one has a minimal disk whose boundary lies near the boundary of a convex body
$E\subset\mathbb R^3$.  If $E'$ is a slightly larger convex body, then one can
deform the disk so that the new boundary lies near the boundary of $E'$, while
the new annular part of the disk remains outside most of $E$.  Repeating this
with an exhaustion
\[
        B_{r_1}\Subset B_{r_2}\Subset B_{r_3}\Subset\cdots
\]
forces the limiting surface to leave every compact subset of $\mathbb R^3$.
This is the mechanism that gives properness.

The technical tool behind the Mart\'{\i}n--Morales lemma is a L\'opez--Ros deformation
combined with a labyrinth.  The L\'opez--Ros deformation is a classical way of
changing the Weierstrass data of a minimal surface while keeping the surface
minimal; see \cite{LopezRos1991}.  A labyrinth is a carefully chosen collection
of thin compact obstacles in the parameter domain.  The deformation makes the
intrinsic metric very large on these obstacles.  Consequently, any path trying
to cross the corresponding annular region must become long.  This idea goes
back to the constructions of Jorge--Xavier \cite{JX} and was
developed further in many later constructions of complete bounded or proper
minimal surfaces; see for instance \cite{N,MM04,MM05}.

In the present paper, however, we use the L\'opez--Ros labyrinth only for the
radial control supplied by \cite[Lemma~2]{MM05}.  We do not take our area packet
from the labyrinth itself.  This distinction is important.  A L\'opez--Ros
deformation can make the intrinsic metric very large on a labyrinth piece, but
it can also move the image of that piece in space.  Thus such a piece is not, by
itself, a safe source of area inside a prescribed ball.

The second ingredient is the Riemann--Hilbert method for conformal minimal
immersions.  Alarc\'on--Drinovec Drnov\v{s}ek--Forstneri\v{c}--L\'opez proved an
approximate Riemann--Hilbert theorem for conformal minimal immersions
\cite[Theorem~3.6]{ADDFL}; see also \cite{AlarconForstneric2015}.
We use the consequence stated in \cite[Lemma~4.1]{ADDFL}: a conformal minimal
immersion of a compact bordered Riemann surface can be approximated uniformly
while making the intrinsic distance from a fixed interior point to the boundary
arbitrarily large.  Since the approximation is uniform in $\mathbb R^3$, if the image of the original compact surface lies a positive distance inside a
ball, then the image of the perturbed surface still lies inside that ball.

Finally, large intrinsic distance across an annulus forces large area.  This is
an elementary extremal-length observation: if every curve crossing a conformal
annulus has large length for the induced metric, then the area of that annulus
for the same metric must be large.  Thus the Riemann--Hilbert perturbation gives
a compact set $L$ with
\[
        \int_L dA\ge a_0
\]
while keeping its image inside the prescribed ball.

The induction therefore has two separate roles.  The Mart\'{\i}n--Morales step pushes
the surface outward and gives properness.  The Riemann--Hilbert step, performed
afterward and with $C^0$ control, creates a definite area packet without losing
the radial control.  By choosing the radii $r_n$ to grow like $\log n$, the ball
$B_R$ contains exponentially many such packets, and this gives the lower bound
\[
        A_X(R)\ge Ce^{cR}.
\]

\subsection{Some notation.} We shall use the following notation.  If $A$ and $B$ are subsets of a Riemann
surface, then
\[
        A\Subset B
\]
means that the closure of $A$ is compact and contained in $B$.  A polygon $P$
will always mean a polygonal disk in the complex plane
\[
        \mathbb C\simeq \mathbb R^2.
\]
For $\varepsilon>0$, we write $P^\varepsilon$ for the inner parallel polygon
obtained by moving the sides of $P$ a distance $\varepsilon$ toward the interior;
in particular,
\[
        \operatorname{Int}P^\varepsilon\Subset \operatorname{Int}P,
\]
where $\operatorname{Int}P$ denotes the interior of the region bounded by the
polygon $P$.

Similarly, if $E\subset\mathbb R^3$ is a regular convex body, then $E_{-a}$
denotes the inner parallel body at distance $a$ from $\partial E$.

For a conformal minimal immersion
\[
        X:\Omega\longrightarrow \mathbb R^3,
\]
we write
\[
        A_X(R)
        =
        \int_{X^{-1}(B_R)}dA_X
\]
for the parametrized area of the portion of the surface contained in the
Euclidean ball $B_R$.  Thus area is counted with multiplicity.

\subsection{The Mart\'{\i}n--Morales convex-body step} \label{subsec:23}

We recall the part of \cite[Lemma~2]{MM05} that we shall use.  Let
$E\Subset E'$ be regular bounded convex domains in $\mathbb R^3$, both
containing the origin.  Let $P\subset\mathbb C$ be a polygonal disk and let
\[
        X:\operatorname{Int}P\longrightarrow \mathbb R^3
\]
be a conformal minimal immersion with $X(0)=0$.  Let
$\varepsilon,a,b>0$.  Assume that
\[
        X\bigl(
        \operatorname{Int}P\setminus\operatorname{Int}P^\varepsilon
        \bigr)
        \subset
        E\setminus E_{-a},
\]
and that
\[
        \sqrt{
        \left(a+\frac1{\kappa_2(\partial E)}\right)^2+a^2
        }
        -
        \frac1{\kappa_2(\partial E)}
        +
        \varepsilon
        <
        \frac{\operatorname{dist}(\partial E,\partial E')}{2}.
\]
Here $\kappa_2(\partial E)$ is the maximum of the largest principal curvature of
$\partial E$, in the notation of \cite{MM05}.

Then \cite[Lemma~2]{MM05} gives a polygon $Q$ and a conformal minimal immersion
\[
        Y:\operatorname{Int}Q\longrightarrow \mathbb R^3,
        \qquad
        Y(0)=0,
\]
with the following properties.

\begin{enumerate}
\item[(MM1)]
\[
        \overline{\operatorname{Int}P^\varepsilon}
        \subset
        \operatorname{Int}Q
        \subset
        \overline{\operatorname{Int}Q}
        \subset
        \operatorname{Int}P.
\]

\item[(MM2)] If $\sigma>0$ is defined by
\[
        \sqrt{
        \left(a+\frac1{\kappa_2(\partial E)}\right)^2
        +
        (2\sigma+a)^2
        }
        -
        \frac1{\kappa_2(\partial E)}
        +
        \varepsilon
        =
        \frac{\operatorname{dist}(\partial E,\partial E')}{2},
\]
then
\[
        \operatorname{dist}_{Y}
        \bigl(z,P^\varepsilon\bigr)
        >
        \sigma
        \qquad\text{for every } z\in Q.
\]

\item[(MM3)]
\[
        Y(Q)\subset E'\setminus E'_{-b}.
\]

\item[(MM4)]
\[
        Y\bigl(
        \operatorname{Int}Q\setminus\operatorname{Int}P^\varepsilon
        \bigr)
        \subset
        \mathbb R^3\setminus E_{-2(a+b)}.
\]

\item[(MM5)]
\[
        \|Y-X\|<\varepsilon
        \qquad\text{on } \operatorname{Int}P^\varepsilon.
\]
\end{enumerate}
\begin{Rem} \label{re:ed}In the application below we shall use the harmless strengthened form in which
the approximation size in (MM5) is prescribed independently of the width
\(\varepsilon\) of the inner parallel polygon.  This follows directly from the
proof of \cite[Lemma~2]{MM05}: after \(\varepsilon\), \(a\), and \(b\) are fixed,
the L\'opez--Ros parameter may be chosen so that the approximation on
\(\operatorname{Int}P^\varepsilon\) is smaller than any prescribed number.
Thus, in (MM5), the right-hand side may be replaced by an arbitrary
\(\delta>0\). \end{Rem}

We shall use (MM1), (MM3), (MM4), and (MM5).  The distance estimate (MM2) is not
needed to prove completeness here, because the final immersion will be proper in
$\mathbb R^3$, and proper immersions into the complete ambient space
$\mathbb R^3$ are complete.

Since $E'$ is convex, (MM3) implies by the convex hull property that
\[
        Y(\operatorname{Int}Q)\subset E'.
\]
Indeed, every supporting halfspace of $E'$ contains $Y(Q)$, and the coordinate
functions of $Y$ are harmonic.

\subsection{Extremal length of an annulus}

We recall the elementary extremal-length fact used below; see, for instance,
\cite[Chapter~4]{Ahlfors}.  Let
\[
        A(r,R)=\{z\in\mathbb C: r<|z|<R\},
        \qquad
        0<r<R,
\]
and let $\Gamma$ be the family of curves in $A(r,R)$ joining the two boundary
components.  With the convention
\begin{equation} \label{def:extremal}
        \operatorname{Ext}(\Gamma)
        =
        \sup_{\rho}
        \frac{
        \left(
        \inf_{\gamma\in\Gamma}\int_\gamma \rho\,|dz|
        \right)^2
        }{
        \int_{A(r,R)}\rho^2\,dA
        },
\end{equation}
where the supremum is over all nonnegative Borel functions $\rho$ with
$0<\int_{A(r,R)}\rho^2\,dA<\infty$, 
one has
\[
        \operatorname{Ext}(\Gamma)
        =
        \frac{2\pi}{\log(R/r)}.
\]
Equivalently, if
\[
        \operatorname{mod} A(r,R)
        =
        \frac{1}{2\pi}\log\frac{R}{r}
\]
denotes the conformal modulus of the annulus, with our convention, then
\[
        \operatorname{Ext}(\Gamma)
        =
        \frac{1}{\operatorname{mod} A(r,R)}.
\]

More generally, if $\mathcal A$ is any conformal annulus and $\Gamma$ is the
family of curves joining its two boundary components, then
$\operatorname{Ext}(\Gamma)$ is a positive finite number depending only on the
conformal type of $\mathcal A$.  In particular, if a conformal metric
\[
        ds^2=\lambda^2|dz|^2
\]
on $\mathcal A$ has the property that every curve in $\Gamma$ has length at
least $\Lambda$, then the definition of extremal length gives
\[
        \operatorname{Ext}(\Gamma)
        \ge
        \frac{\Lambda^2}{\int_{\mathcal A}\lambda^2\,dA}.
\]
Therefore
\[
        \int_{\mathcal A}\lambda^2\,dA
        \ge
        \frac{\Lambda^2}{\operatorname{Ext}(\Gamma)}.
\]
Thus, on a fixed conformal annulus, forcing all crossing curves to have large
length forces the area of the annulus to be large.  This is the only property
of extremal length used in the sequel.

\subsection{A controlled area-inflation lemma}

The Riemann--Hilbert input that we use  is \cite[Lemma~4.1]{ADDFL}.  This lemma is proved
using the Riemann--Hilbert theorem for conformal minimal immersions
\cite[Theorem~3.6]{ADDFL}. We denote by
\[
        \mathring M=M\setminus\partial M
\]
the interior of the compact bordered Riemann surface $M$.

 In the form needed here, it says that if $M$ is a
compact bordered Riemann surface, $G:M\to\mathbb R^n$ is a conformal minimal
immersion, $p_0\in \mathring M$, and $\Lambda>0$, then $G$ can be approximated
arbitrarily closely in the $C^0(M)$ topology by conformal minimal immersions
$\widehat G$ satisfying
\[
        \operatorname{dist}_{\widehat G}(p_0,\partial M)>\Lambda.
\]
Moreover, since conformal minimal immersions are harmonic, the approximation is
in the $C^r$ topology on compact subsets of $\mathring M$.

\begin{Lem}[Controlled area inflation]
\label{lem:controlled-area-inflation}
Let $M$ be a compact bordered Riemann surface and let
\[
        X:M\longrightarrow \mathbb R^3
\]
be a conformal minimal immersion of class $C^1(M)$.  Let
$K\Subset \mathring M$ be a smoothly bounded compact disk such that
\[
        \mathcal A=\overline{M\setminus K}
\]
is an annulus.  Let
$U\subset\mathbb R^3$ be an open set such that
\[
        X(M)\Subset U.
\]
Then, for every $A_0>0$ and every $\eta>0$ satisfying
\[
        0<\eta<
        \operatorname{dist}\bigl(X(M),\mathbb R^3\setminus U\bigr),
\]
there exist a conformal minimal immersion
\[
        \widetilde X:M\longrightarrow \mathbb R^3
\]
and a compact set
\[
        L\Subset M\setminus K
\]
such that
\[
        \|\widetilde X-X\|_{C^0(M)}<\eta,
\]
\[
        \widetilde X(L)\subset U,
\]
and
\[
        \int_L dA_{\widetilde X}\ge A_0.
\]
Furthermore, the approximation may be chosen $C^1$-close to $X$ on any fixed
compact subset of $\mathring M$.
\end{Lem}

\begin{proof}
Choose a point
\[
        p_0\in \mathring K.
\]

Let $\Gamma$ be the family of curves in the annulus
\[
        \mathcal A=\overline{M\setminus K}
\]
joining $\partial K$ to $\partial M$.  Its extremal length $\operatorname{Ext}(\Gamma)$
is a positive finite number depending only on the conformal annulus
$\mathcal A$. 

Choose $D>0$ such that every point of $\partial K$ can be joined to $p_0$ by a
curve in $K$ of $X$-length at most $D/2$.  Since the approximation supplied by
\cite[Lemma~4.1]{ADDFL} is $C^1$ on compact subsets of $\mathring M$, after
choosing it sufficiently close on $K$ the same property holds for
$\widetilde X$ with $D$ in place of $D/2$.

Choose
\[
        \Lambda
        >
        D+\sqrt{2A_0\,\operatorname{Ext}(\Gamma)}.
\]
By \cite[Lemma~4.1]{ADDFL}, there exists a conformal minimal immersion
\[
        \widetilde X:M\longrightarrow \mathbb R^3
\]
such that
\[
        \|\widetilde X-X\|_{C^0(M)}<\eta
\]
and
\[
        \operatorname{dist}_{\widetilde X}(p_0,\partial M)>\Lambda.
\]
We also choose the approximation $C^1$-close to $X$ on $K$.

We claim that every curve in $\mathcal A$ joining $\partial K$ to $\partial M$ has
$\widetilde X$-length at least $\Lambda-D$.  Indeed, let
$\gamma\subset\mathcal A$ be such a curve, with initial point
$q\in\partial K$ and final point in $\partial M$.  Join $p_0$ to $q$ by a curve
$\alpha\subset K$ of $\widetilde X$-length at most $D$.  Then
$\alpha*\gamma$ joins $p_0$ to $\partial M$, and hence
\[
        \operatorname{length}_{\widetilde X}(\alpha)
        +
        \operatorname{length}_{\widetilde X}(\gamma)
        >
        \Lambda.
\]
Therefore
\[
        \operatorname{length}_{\widetilde X}(\gamma)
        >
        \Lambda-D.
\]

Since $\widetilde X$ is conformal, the induced metric on $\mathcal A$ has the
form
\[
        ds_{\widetilde X}^2
        =
        \lambda_{\widetilde X}^2 |dz|^2.
\]
By the definition of extremal length and taking
$\rho=\lambda_{\widetilde X}$, we get
\[
        \operatorname{Ext}(\Gamma)
        \ge
        \frac{
        (\Lambda-D)^2
        }{
        \int_{\mathcal A}\lambda_{\widetilde X}^2\,dA
        }.
\]
Hence
\[
        \int_{\mathcal A} dA_{\widetilde X}
        =
        \int_{\mathcal A}\lambda_{\widetilde X}^2\,dA
        \ge
        \frac{(\Lambda-D)^2}{\operatorname{Ext}(\Gamma)}
        >
        2A_0.
\]
By inner regularity of the area measure, there exists a compact set
\[
        L\Subset M\setminus K
\]
such that
\[
        \int_L dA_{\widetilde X}\ge A_0.
\]

Finally, the $C^0$ estimate and the choice of $\eta$ imply
\[
        \widetilde X(M)\subset U.
\]
In particular,
\[
        \widetilde X(L)\subset U.
\]
This proves the lemma.
\end{proof}

\begin{Rem}
The preceding lemma does not claim that \cite[Lemma~4.1]{ADDFL} directly
produces a two-dimensional area packet.  The ADDFL lemma produces large
intrinsic boundary distance while preserving $C^0$ control.  The area estimate
then follows from extremal length on the annulus $\overline{M\setminus K}$.
\end{Rem}

\subsection{An area-enriched Mart\'{\i}n--Morales step}

We now combine \cite[Lemma~2]{MM05} with
Lemma~\ref{lem:controlled-area-inflation}.

\begin{Lem}[Area-enriched convex-body step]
\label{lem:area-enriched-step}
Let $E\Subset E'$ be regular bounded convex domains in $\mathbb R^3$, both
containing the origin.  Let $P\subset\mathbb C$ be a polygonal disk and let
\[
        X:\operatorname{Int}P\longrightarrow \mathbb R^3
\]
be a conformal minimal immersion with $X(0)=0$.  Let $K_0\Subset\operatorname{Int}P$
be compact, and let $\eta_0>0$ and $A_0>0$ be given.

Choose $a,b>0$ and then choose $\varepsilon>0$ so small that
\[
        K_0\Subset \operatorname{Int}P^\varepsilon,
\]
\[
        X\bigl(
        \operatorname{Int}P\setminus\operatorname{Int}P^\varepsilon
        \bigr)
        \subset
        E\setminus E_{-a},
\]
and
\[
        \sqrt{
        \left(a+\frac1{\kappa_2(\partial E)}\right)^2+a^2
        }
        -
        \frac1{\kappa_2(\partial E)}
        +
        \varepsilon
        <
        \frac{\operatorname{dist}(\partial E,\partial E')}{2}.
\]
Write
\[
        D=\operatorname{Int}P^\varepsilon.
\]
Then there exist a polygon $Q$, a conformal minimal immersion
\[
        Y:\operatorname{Int}Q\longrightarrow \mathbb R^3,
        \qquad
        Y(0)=0,
\]
and a compact set
\[
        L\Subset \operatorname{Int}Q\setminus \overline D
\]
such that
\[
       \overline D
        \subset
        \operatorname{Int}Q
        \subset
        \overline{\operatorname{Int}Q}
        \subset
        \operatorname{Int}P,
\]

\[
        Y(Q)\subset E'\setminus E'_{-b},
\]

\[
        Y(\operatorname{Int}Q)\subset E',
\]

\[
        Y\bigl(
        \operatorname{Int}Q\setminus D
        \bigr)
        \subset
        \mathbb R^3\setminus E_{-3(a+b)},
\]
\[
        \|Y-X\|_{C^0(D)}<\eta_0,
\]
\[
        \|Y-X\|_{C^1(K_0)}<\eta_0,
\]
and
\[
        \int_L dA_Y\ge A_0.
\]
\end{Lem}
\begin{proof}
Choose a compact domain $K_1$ with
\[
        K_0\Subset K_1\Subset D.
\]
Choose $\delta>0$ small enough that $\delta<\eta_0/3$ and such that
$C^0$-control by $\delta$ on $K_1$ gives $C^1$-control by $\eta_0/3$ on
$K_0$ through the standard interior estimates for harmonic functions.  We first
apply \cite[Lemma~2]{MM05}, in the strengthened form described in
Remark~\ref{re:ed}, to the data
\[
        E\Subset E',
        \qquad
        X:\operatorname{Int}P\to\mathbb R^3,
\]
with the parameters $\varepsilon,a,b/4$ and approximation size $\delta$.  This
gives a polygon $Q_0$ and a conformal minimal immersion
\[
        Y_0:\operatorname{Int}Q_0\longrightarrow \mathbb R^3,
        \qquad
        Y_0(0)=0,
\]
satisfying (MM1)--(MM5).   In particular, by (MM1) one has
\[
        \overline D
        \subset
        \operatorname{Int}Q_0
        \subset
        \overline{\operatorname{Int}Q_0}
        \subset
        \operatorname{Int}P.
\]
By (MM3) and the convex hull property,
\[
        Y_0(\operatorname{Int}Q_0)\subset E'.
\]
By (MM4),
\[
        Y_0\bigl(
        \operatorname{Int}Q_0\setminus D
        \bigr)
        \subset
        \mathbb R^3\setminus E_{-2(a+b/4)}.
\]
By (MM5), in this strengthened form, we have
\[
        \|Y_0-X\|_{C^0(D)}<\delta<\frac{\eta_0}{3}.
\]
By the choice of $\delta$ and the fact that $K_0\Subset K_1\Subset D$, we also
have
\[
        \|Y_0-X\|_{C^1(K_0)}<\frac{\eta_0}{3}.
\]

Choose a polygon $Q_1$ with
\[
        \overline D
        \subset
        \operatorname{Int}Q_1
        \subset
        \overline{\operatorname{Int}Q_1}
        \subset
        \operatorname{Int}Q_0.
\]
Recall that we applied \cite[Lemma~2]{MM05} with the boundary-collar parameter
\(b/4\).  Hence
\[
        Y_0(Q_0)\subset E'\setminus E'_{-b/4}.
\]
We choose \(Q_1\) sufficiently close to \(Q_0\) so that
\[
        Y_0(Q_1)\subset E'\setminus E'_{-b/2}
\]
and
\[
        Y_0\bigl(\overline{\operatorname{Int}Q_1}\bigr)\Subset E'.
\]
Set
\[
        M=\overline{\operatorname{Int}Q_1}.
\]
Thus $M$ is a compact bordered Riemann surface.  Choose a smoothly bounded
compact disk $K$ such that
\[
        K_1\cup \overline D
        \Subset
        K
        \Subset
        \operatorname{Int}Q_1,
\]
and such that
\[
        \mathcal A=\overline{M\setminus K}
\]
is an annulus.  Choose $p_0\in \operatorname{Int}K$.

Since
\[
        Y_0(M)\Subset E',
\]
there is a positive number
\[
        d_0=
        \operatorname{dist}\bigl(Y_0(M),\mathbb R^3\setminus E'\bigr)>0.
\]
We shall apply Lemma~\ref{lem:controlled-area-inflation} to
\[
        Y_0|_M:M\longrightarrow \mathbb R^3,
        \qquad
        K\Subset \mathring M,
        \qquad
        U=E',
\]
with the prescribed number $A_0$ and with approximation size chosen below.

Choose $\eta>0$ smaller than $d_0$, $\eta_0$, $b/2$, and the radial margin
needed to preserve the conclusion of (MM4).  In particular, choose
$\eta<a+b$.  Applying Lemma~\ref{lem:controlled-area-inflation} with
approximation size $\eta/2$ gives a conformal minimal immersion
\[
        \widehat Y:M\longrightarrow \mathbb R^3
\]
and a compact set
\[
        L\Subset M\setminus K
        \quad \mbox{and thus} \quad L \Subset 
        \operatorname{Int}Q_1\setminus \overline D
\]
such that
\[
        \|\widehat Y-Y_0\|_{C^0(M)}<\frac{\eta}{2},
\]
\[
        \widehat Y(L)\subset E',
\]
and
\[
        \int_L dA_{\widehat Y}\ge A_0.
\]
The approximation is also $C^1$-small on $K$, and hence on $K_0$.

Now set
\[
        \widetilde Y=\widehat Y-\widehat Y(0).
\]
Since $Y_0(0)=0$, the translation vector has norm less than $\eta/2$.  Hence
\[
        \|\widetilde Y-Y_0\|_{C^0(M)}<\eta.
\]
The translation does not change the induced metric, the area of $L$, or
minimality, and it gives
\[
        \widetilde Y(0)=0.
\]
Moreover, since the final $C^0$-change from $Y_0$ is less than $\eta<d_0$, we
still have
\[
        \widetilde Y(M)\subset E'.
\]
In particular, $\widetilde Y(L)\subset E'$ and
\[
        \int_L dA_{\widetilde Y}\ge A_0.
\]

Choosing $\eta$ sufficiently small gives
\[
        \|\widetilde Y-X\|_{C^0(D)}<\eta_0
\]
and
\[
        \|\widetilde Y-X\|_{C^1(K_0)}<\eta_0.
\]

The same $C^0$-smallness preserves the radial exclusion coming from (MM4), with
slightly weaker constants.  More precisely, choosing $\eta<a+b$ gives
\[
        \widetilde Y\bigl(
        \operatorname{Int}Q_1\setminus D
        \bigr)
        \subset
        \mathbb R^3\setminus E_{-3(a+b)}.
\]

Since the final $C^0$-change from $Y_0$ is less than $\eta<b/2$, the distance
to $\partial E'$ can increase by at most $b/2$.  Thus
\[
        \widetilde Y(Q_1)\subset E'\setminus E'_{-b}.
\]

Set
\[
        Q=Q_1,
        \qquad
        Y=\widetilde Y|_{\operatorname{Int}Q}.
\]

Then $Y$, $Q$, and $L$ satisfy the desired conclusions.
\end{proof}

We construct inductively polygons $P_n$ and conformal minimal immersions
\[
        X_n:\operatorname{Int}P_n\longrightarrow \mathbb R^3,
        \qquad
        X_n(0)=0.
\]
We also choose polygonal domains
\[
        D_n=\operatorname{Int}P_n^{\varepsilon_n}
\]
and compact protected domains $\Omega_n$ such that
\[
        \Omega_n\Subset D_n\Subset \Omega_{n+1}.
\]
The limiting Riemann surface will be
\[
        \Omega=\bigcup_{n=1}^{\infty}\Omega_n
        =
        \bigcup_{n=1}^{\infty}D_n.
\]

The induction is arranged so that the following properties hold.

\begin{enumerate}
\item[(I1)]
\[
        \Omega_n\Subset D_n\Subset \Omega_{n+1}
        \Subset D_{n+1}
        \Subset \operatorname{Int}P_{n+1}.
\]

\item[(I2)]
\[
        X_n\bigl(
        \operatorname{Int}P_n\setminus D_n
        \bigr)
        \subset
        B_{r_n}\setminus B_{r_n-a_n}.
\]

\item[(I3)]
\[
        X_{n+1}
        \bigl(
        \operatorname{Int}P_{n+1}\setminus D_n
        \bigr)
        \subset
        \mathbb R^3\setminus B_{r_n-\tau_n},
\]
where
\[
        \tau_n=3(a_n+b_n).
\]

\item[(I4a)]
\[
        \|X_{n+1}-X_n\|_{C^0(D_n)}<\eta_n.
\]

\item[(I4b)]
\[
        \|X_{n+1}-X_n\|_{C^1(\Omega_n)}<\eta_n.
\]

\item[(I5)] There exists a compact set
$
        L_n\Subset \Omega_{n+1}\setminus \overline D_n
$
\[
        \int_{L_n}dA_{X_{n+1}}\ge 2a_0,
\]
where $a_0>0$ is fixed independently of $n$.
\end{enumerate}
\begin{Rem} \label{re:rho}By the choice of the parameters, the sequence
\[
        \rho_n
        :=
        r_n-\tau_n-\sum_{k\ge n}\eta_k
\]
is eventually increasing and tends to infinity. \end{Rem}

\vskip1mm

\subsection{Choice of radii and parameters}

Let
\[
        r_n=\log(n+n_0),
        \qquad
        E_n=B_{r_n},
\]
where \(n_0\) is chosen sufficiently large.  Put
\[
        \mu_n=r_{n+1}-r_n.
\]
Then
\[
        \mu_n
        =
        \log\left(1+\frac1{n+n_0}\right)
        \sim
        \frac1{n+n_0}.
\]
After increasing \(n_0\) if necessary, we may assume
\[
        \frac12\mu_n\le \mu_{n+1}\le \mu_n
        \qquad\text{for all }n.
\]

Choose any positive sequence \(a_n\) satisfying
\[
        a_n<\frac{\mu_n}{100}.
\]
Next choose \(b_n>0\) such that
\[
        b_n<\frac{a_{n+1}}{20}.
\]
Define the upper bounds \(\bar\eta_n > 0\) by
\[
        \bar\eta_n =  2^{-n}\frac{\mu_n}{100} 
         \, ,
\]
so using that $\mu_n$ is decreasing gives that 
\[
	\sum_{k \geq n} \bar\eta_k \leq 	\sum_{k \geq n} \, \left( 2^{-k} \, \frac{\mu_n}{100} \right) \leq \frac{\mu_n}{100} \, .
\]

The numbers \(\varepsilon_n\) and \( \eta_n \) will be chosen during the induction.  We shall
always require
\[
        0<\varepsilon_n<\frac{\mu_n}{100}.
\]

\[
        0<\eta_n<\min\left\{\bar\eta_n,\frac{a_{n+1}}{20}\right\}.
\]
This guarantees that the numerical hypothesis in \cite[Lemma~2]{MM05} can be
satisfied once \(a_n\) has been chosen sufficiently small compared with
\(\mu_n\).

Furthermore, after \(X_n\) and \(\Omega_n\) have been constructed, we choose the future
tolerances so that
\[
        \sum_{k\ge n}\eta_k
        <
        \frac12
        \min_{\Omega_n}
        \left|
        \frac{\partial X_n}{\partial x}
        \right|.
\]


\begin{Rem} \label{re:areapre}
We would like to point out that we also protect the area of the packet \(L_n\) for all later stages.  Since
\[
        \int_{L_n}dA_{X_{n+1}}\ge 2a_0
\]
and \(L_n\) is compact, there exist a compact neighborhood
\[
        U_n\Subset \Omega_{n+1}\setminus \overline{D_n}
\]
of \(L_n\) and a number \(\alpha_n>0\) such that, if \(Z\) is a conformal
immersion on \(U_n\) satisfying
\[
        \|Z-X_{n+1}\|_{C^1(U_n)}<\alpha_n,
\]
then
\[
        \int_{L_n}dA_Z\ge a_0.
\]
For this reason, the future tolerances \(\eta_k\), \(k\ge n+1\), are chosen so that
\[
        \sum_{k\ge n+1}\eta_k<\alpha_n. \]
\end{Rem}

\subsection{The inductive construction}

We start by choosing $P_1, D_1, X_1$, and $\Omega_1$.
Choose $P_1$ to be a polygonal disk in $\mathbb C$ whose boundary lies in the
annulus
\[
        \{r_1-a_1/2<|z|<r_1\} \quad \mbox{and} \quad 0 \in \operatorname{Int} P_1.
\]
Then, since $X_1(z)=(z,0)$, choosing $\varepsilon_1>0$ sufficiently small gives
\[
        X_1\bigl(
        \operatorname{Int}P_1\setminus
        \operatorname{Int}P_1^{\varepsilon_1}
        \bigr)
        \subset
        B_{r_1}\setminus B_{r_1-a_1}.
\]
Thus (I2) holds for $n=1$. Let $\Omega_1$ be a compact subdomain of $D_1 = \Int P_1^{\varepsilon_1}$ containing $0$.

\vskip1mm
We now explain the inductive construction.
Suppose that $X_n$, $P_n$, $D_n$, and $\Omega_n$ have been constructed. We must now construct $X_{n+1}$, $P_{n+1}$, $D_{n+1}$ (i.e., choose $\varepsilon_{n+1}$), and $\Omega_{n+1}$.

We first choose \(\eta_n>0\) small enough so that
\[
        0<\eta_n<\min\left\{\bar\eta_n,\frac{a_{n+1}}{20}\right\},
\]
and so that all previously imposed tail conditions are preserved.  In
particular, the remaining future tolerances can still be chosen to preserve
nondegeneracy of the conformal factor on the protected domains and to preserve
the area of all previously constructed packets.  This is possible because, at
the \(n\)-th stage, only finitely many such conditions have been imposed.

We apply Lemma~\ref{lem:area-enriched-step} with
\[
        P=P_n , \quad D=D_n , \quad X=X_n , \quad E=B_{r_n},
        \quad
        E'=B_{r_{n+1}},
        \quad  K_0=\Omega_n \, , 
\]
and with the constants
\[
        A_0=2a_0, \quad a = a_n, \quad b=b_n, \quad \eta_0 = \eta_n.
\]

The shell hypothesis required by \cite[Lemma~2]{MM05} is exactly (I2), and the
remaining numerical hypothesis is guaranteed by the parameter choice in the
previous subsection.  

Lemma~\ref{lem:area-enriched-step} gives a polygon $P_{n+1}$, a map $X_{n+1}$ on the interior of $P_{n+1}$, and a compact set 
$L_n \subset \Int P_{n+1} \setminus \overline{D_n}$ with the following properties:
\begin{enumerate}
\item $\overline{D_n} \subset \Int P_{n+1}$ and $\overline{\Int P_{n+1}} \subset \Int P_n$.
\item $X_{n+1} (P_{n+1}) \subset B_{r_{n+1}} \setminus B_{r_{n+1}-b_n}$ and thus $X_{n+1} (\Int P_{n+1}) \subset B_{r_{n+1}} $. 
\item   $X_{n+1} (\Int P_{n+1} \setminus D_n) \subset \RR^3 \setminus B_{r_{n}- 3 \, (a_n + b_n)} $. 
\item  $ \|X_{n+1}-X_n\|_{C^0(D_n)}<\eta_n$ and  $\|X_{n+1}-X_n\|_{C^1(\Omega_n)}<\eta_n$. 
\item  $ \int_{L_n}dA_{X_{n+1}}\ge 2a_0$. 
\end{enumerate}
It remains to choose $\varepsilon_{n+1}$ (and thus also $D_{n+1}$), to choose $\Omega_{n+1}$, and to verify that  (I1) and (I3)--(I5) hold at $n$ and that (I2) holds at $n+1$.

The first conclusion in (2), the fact that $b_n \leq \frac{a_{n+1}}{20}$, and the continuity of $X_{n+1}$ imply that there exists $\varepsilon_{n+1}' > 0$ so that 
$$
	X_{n+1} \left(\Int P_{n+1} \setminus \Int P_{n+1}^{\varepsilon_{n+1}'} \right) \subset B_{r_{n+1}} \setminus B_{r_{n+1} - a_{n+1}} \, .
$$
Using that $\overline{D_n} \cup L_n \subset \Int P_{n+1}$ is compact, we can 
choose $\varepsilon_{n+1} \in (0 , \varepsilon_{n+1}')$ small enough that 
$$
	L_n \cup \overline{D_n} \subset \Int P_{n+1}^{\varepsilon_{n+1}} \equiv D_{n+1}\, .
$$
Note that this is part of (I5). 
Then choose a compact $\Omega_{n+1}$ with 
$$
	\overline{D_n} \cup L_n \Subset \Omega_{n+1} \Subset   D_{n+1} \, .
$$
This gives (I1) at $n$ and (I2) at $n+1$; (I3) at $n$ follows from (3). (I4a) and (I4b) follow at $n$ by (4).  Finally, the choice of $\varepsilon_{n+1}$ and (5) give (I5) at $n$.

\subsection{The limiting immersion}

By (I4b) and the summability of $\eta_n$, the sequence $X_n$ converges in the
$C^1$ topology on compact subsets of
\[
        \Omega=\bigcup_{n=1}^{\infty}\Omega_n
\]
to a map
\[
        X:\Omega\longrightarrow\mathbb R^3.
\]
We now prove that the limit is an {\bf immersion}.  This is another place where the
summability of the later $C^1$-errors is used.

At the end of the $n$-th stage, the map
\[
        X_n:\operatorname{Int}P_n\longrightarrow \mathbb R^3
\]
is a conformal immersion, and
\[
        \Omega_n\Subset \operatorname{Int}P_n.
\]
Hence its conformal factor is bounded below on $\Omega_n$.  Writing
$z=x+iy$, set
\[
        \lambda_n(p)
        =
        \left|
        \frac{\partial X_n}{\partial x}(p)
        \right|
        =
        \left|
        \frac{\partial X_n}{\partial y}(p)
        \right|,
        \qquad
        p\in \Omega_n.
\]
Since \(X_n\) is an immersion and \(\Omega_n\) is compact, we have
\[
        m_n:=\min_{\Omega_n}\lambda_n>0.
\]
When the later approximation tolerances were chosen, we required
\[
        \sum_{k\ge n}\eta_k<\frac{m_n}{2}
\]
for every \(n\).
This is possible by choosing the numbers \(\eta_k\) inductively and decreasing
them when necessary.  At every finite stage only finitely many previously
introduced conditions have to be respected, and the remaining tail can always be
made as small as desired.

Let now \(p\in\Omega\).  Choose \(n\) such that \(p\in\Omega_n\).  Since the
domains are nested, \(p\in\Omega_k\) for every \(k\ge n\).  By (I4b),
\[
        \|X_{k+1}-X_k\|_{C^1(\Omega_n)}
        \le
        \|X_{k+1}-X_k\|_{C^1(\Omega_k)}
        <
        \eta_k
        \qquad\text{for every } k\ge n.
\]
Therefore
\[
        \left|
        \frac{\partial X}{\partial x}(p)
        -
        \frac{\partial X_n}{\partial x}(p)
        \right|
        \le
        \sum_{k\ge n}
        \left\|
        X_{k+1}-X_k
        \right\|_{C^1(\Omega_n)}
        <
        \sum_{k\ge n}\eta_k
        <
        \frac{m_n}{2}.
\]
It follows that
\[
        \left|
        \frac{\partial X}{\partial x}(p)
        \right|
        \ge
        \left|
        \frac{\partial X_n}{\partial x}(p)
        \right|
        -
        \left|
        \frac{\partial X}{\partial x}(p)
        -
        \frac{\partial X_n}{\partial x}(p)
        \right|
        >
        m_n-\frac{m_n}{2}
        =
        \frac{m_n}{2}
        >
        0.
\]
Thus \(dX_p\neq 0\).

Since each \(X_n\) is conformal and the convergence is \(C^1\) on compact
subsets, the conformality equations pass to the limit:
\[
        \left\langle \frac{\partial X}{\partial x},\frac{\partial X}{\partial y}\right\rangle=0,
        \qquad
        \left|\frac{\partial X}{\partial x}\right|^2=\left|\frac{\partial X}{\partial y}\right|^2.
\]
Since \(\frac{\partial X}{\partial x}(p)\neq0\), these equations imply that \(\frac{\partial X}{\partial y}(p)\neq0\) and that
\(\frac{\partial X}{\partial x}(p)\) and \(\frac{\partial X}{\partial y}(p)\) are linearly independent.  Hence \(dX_p\) has rank
two.  Since \(p\in\Omega\) was arbitrary, \(X\) is an immersion.

The coordinate functions of the \(X_n\)'s are harmonic.  Since \(X_n\to X\)
uniformly on compact subsets, the coordinate functions of \(X\) are harmonic.
Together with conformality and the fact that \(X\) is an immersion, this implies
that \(X\) is a conformal minimal immersion.

  The domain $\Omega$ is simply connected because it is an increasing
union of simply connected planar domains. Moreover, by construction, \(\Omega\) is a bounded simply connected domain in
\(\mathbb C\).  Hence, by the Riemann mapping theorem, \(\Omega\) is conformally
equivalent to the unit disk.

We prove that $X$ is {\bf proper}.  Let $R>0$ be fixed.  Choose $n$ so large that $\rho_n>R$ and so that
$\rho_m\ge \rho_n$ for every $m\ge n$ (see Remark \ref{re:rho}.)
Let $p\in\Omega\setminus D_n$.  Then there exists $m\ge n$ such that
\[
        p\in D_{m+1}\setminus D_m.
\]
By (I3),
\[
        X_{m+1}(p)\in \mathbb R^3\setminus B_{r_m-\tau_m}.
\]
{The remaining tail of the approximations is controlled in the \(C^0\) topology
by (I4a), and hence has size at most
\[
        \sum_{k\ge m+1}\eta_k.
\]
Therefore
\[
        X(p)\in
        \mathbb R^3\setminus
        B_{r_m-\tau_m-\sum_{k\ge m+1}\eta_k}.
\]
Since
\[
        r_m-\tau_m-\sum_{k\ge m+1}\eta_k
        \ge
        r_m-\tau_m-\sum_{k\ge m}\eta_k
        =
        \rho_m
        \ge
        \rho_n
        >
        R,
\]
we get
\[
        X(p)\in \mathbb R^3\setminus B_R.
\]}
Thus
\[
        X(\Omega\setminus D_n)\subset \mathbb R^3\setminus B_R.
\]
It follows that
\[
        X^{-1}(\overline B_R)\subset D_n.
\]
Since $D_n\Subset\Omega$, this proves that $X$ is proper.

Since $X$ is a proper immersion into the complete ambient space $\mathbb R^3$,
the induced metric on $\Omega$ is complete.  Indeed, if $\gamma$ is a divergent
curve in $\Omega$, then properness implies that $|X(\gamma(t))|\to\infty$.
Hence the Euclidean length of $X\circ\gamma$ is infinite, and therefore the
$X$-length of $\gamma$ is infinite.

\subsection{Exponential area growth}

We now prove the desired {area lower bound}.

\begin{Thm}
\label{thm:proper-R3-exponential-area-growth}
There exists a complete proper conformal minimal immersion
\[
        X:\Omega\longrightarrow\mathbb R^3
\]
from a simply connected Riemann surface $\Omega$ such that
\[
        \int_{X^{-1}(B_R)}dA_X
        \ge
        Ce^{cR}
\]
for all sufficiently large $R$ and for some constants $C,c>0$.
\end{Thm}

\begin{proof}
The immersion $X:\Omega\to\mathbb R^3$ has already been constructed above and
shown to be proper and complete.  It remains only to estimate its area growth.

The compact sets $L_n$ are pairwise disjoint because
\[
        L_n\Subset D_{n+1}\setminus \overline D_n
\]
and the domains $D_n$ are nested.

Each $L_n$ is contained in $\Omega_{n+1}$, and all later approximations are
chosen $C^1$-small on $\Omega_{n+1}$ by (I1) and (I4b). In particular, the triangle inequality gives that
\[
	\| X_{n+1} - X \|_{C^1 (L_n)} \leq \sum_{k> n} \eta_k \leq \frac{\mu_n}{100} \, .
\]
Note that this norm goes to zero as $n \to \infty$.  We use this both to control  the areas of $X(L_n)$ and the  locations of the images $X(L_n)$. 
By the area-preservation condition chosen during the induction (Remark \ref{re:areapre}),
\[
        \int_{L_n}dA_X\ge a_0, \quad \mbox{for every \(n\).}
\]

Next, $X_{n+1}(L_n)\subset B_{r_{n+1}}$ implies that 
\[
        X(L_n)\subset
        B_{r_{n+1}+\frac{\mu_n}{100}} \subset B_{r_{n+2}} \, .
\]

Let
\[
        N(R)
        =
        \max\left\{
        n:
        r_{n+2}\le R
        \right\}.
\]
Then
\[
        L_1,\ldots,L_{N(R)}
        \subset
        X^{-1}(B_R).
\]
Therefore
\[
        A_X(R)
        =
        \int_{X^{-1}(B_R)}dA_X
        \ge
        \sum_{n=1}^{N(R)}
        \int_{L_n}dA_X
        \ge
        a_0 \, N(R).
\]
Since
\(
        r_n=\log(n+n_0)
\), 
there exist constants $C_1,c_1>0$ such that
\[
        N(R)\ge C_1e^{c_1R}
\]
for all sufficiently large $R$.  Consequently,
\[
        A_X(R)\ge Ce^{cR}
\]
for some constants $C,c>0$ and all sufficiently large $R$.
\end{proof}

\begin{Rem}[Gaussian lower growth]
The same construction gives Gaussian lower growth if the radii are chosen more
densely.  For example, take
\[
        r_n=\sqrt{\frac1d\log(n+n_0)}
\]
for some fixed $d>0$.  Then
\[
        \#\{n:r_n\le R\}\simeq \exp(dR^2).
\]
Since $r_n\to\infty$ and the gaps $r_{n+1}-r_n$ are positive, the parameters
$a_n,b_n,\varepsilon_n,\eta_n$ can be chosen much smaller than these
gaps and with the same compatibility conditions as above.  The same properness
and packet-counting argument therefore gives
\[
        \int_{X^{-1}(B_R)}dA_X
        \ge
        C\exp(cR^2)
\]
for all sufficiently large $R$.
\end{Rem}
\section{Final comments and further questions}
\label{sec:final-comments}

The examples in this paper show that properness alone does not impose any
reasonable polynomial upper bound for the extrinsic area growth of a minimal
surface.  This should be contrasted with the volume-doubling restrictions for
minimal submanifolds confined in slabs, or in regions that are slab-like at all
large scales, obtained in \cite{CM1}.  Thus rapid area growth is compatible with
properness, but not with the type of confinement considered in \cite{CM1}.

The two constructions are quite different.  In Section~1 the examples are
holomorphic graphs in $\mathbb C^2=\mathbb R^4$.  Their rapid area growth comes
from finding many regions where the holomorphic function remains bounded while
its derivative is large.  In Section~2 the example lies in $\mathbb R^3$, but the
construction is necessarily less explicit.  The Mart\'{\i}n--Morales deformation \cite{MM05} is used to push the boundary
outward and obtain properness, while
the Riemann--Hilbert perturbation from \cite{ADDFL} is used separately to insert
area.  This separation is important: the area packets are not taken from the
L\'opez--Ros transformation itself, but are inserted afterwards by a uniformly small
perturbation.

We end with several natural questions that are not addressed by the present
construction.

\begin{Que}
Can the proper minimal immersion in $\mathbb R^3$ constructed in
Section~2 be replaced by a proper embedded minimal surface with comparable
area growth?
\end{Que}

The construction in Section~2 gives a proper immersion, and the area lower bound
is for parametrized area, counted with multiplicity.  Thus a stronger embedded
version would require additional separation estimates, not just the
properness and area-inflation mechanisms used here.

\begin{Que}
Can one prescribe the extrinsic area growth more precisely for proper minimal
immersions in $\mathbb R^3$?
\end{Que}

The method gives exponential growth, and by choosing the radii more densely it
also gives Gaussian lower growth.  More generally, one can ask how flexibly the
sequence of radii and area packets can be chosen, and what limitations remain
when one requires properness throughout the induction.

\begin{Que}
Can one obtain rapid area growth together with stronger geometric control, such
as embeddedness, curvature estimates, or restrictions on the number of sheets in
extrinsic balls?
\end{Que}

The present construction is designed to force area while preserving properness.
It does not attempt to control self-intersections or multiplicity in extrinsic
balls.  Understanding whether rapid growth can coexist with stronger global
geometric control seems to be a separate problem.

\begin{Que}
What are the sharp geometric hypotheses that force polynomial area growth for
proper minimal submanifolds?
\end{Que}

The results of \cite{CM1} show that slab-type confinement leads to strong
volume restrictions.  The examples here show that, without such confinement,
proper minimal submanifolds may have very large area growth.  It would be
interesting to understand more precisely which large-scale geometric conditions
separate these two behaviors.

\end{document}